\documentclass[showpacs,showkeys,preprintnumbers,amsmath,amssymb]{revtex4}
\usepackage{graphicx}% Include figure files
\usepackage{dcolumn}% Align table columns on decimal point
\usepackage{bm}% bold math
\usepackage{mathrsfs}
\usepackage{amsmath}
\newtheorem{theorem}{Theorem}[section]
\newtheorem{lemma}[theorem]{Lemma}

\newtheorem{remark}[theorem]{Remark}
\newenvironment{proof}[1][Proof.]{\begin{trivlist}
     \item[\hskip \labelsep {\bfseries #1}]}{\end{trivlist}}

\begin{document}

%\preprint{APS/123-QED}

\title{The Two-Point Connection Problem for a Sub-Class of the Heun Equation}% Force line breaks with \\
\author{R. Williams}%
\email{rwilliam@ictp.it}
\affiliation{%
Mathematics Section, \\ International Center for Theoretical Physics, Trieste, Italy 
}%
\author{D. Batic}
\email{davide.batic@uwimona.edu.jm}
\affiliation{%
Department of Mathematics,\\  University of the West Indies, Kingston 6, Jamaica 
}%

\date{\today}% It is always \today, today,
             %  but any date may be explicitly specified

\begin{abstract}
The present article discusses the two point connection problem for Heun's differential equation. We employ a contour integral method to derive connection matrices for a sub-class of the Heun equation containing 3 free parameters. Explicit expressions for the connection coefficients are found.
\end{abstract}

% PACS, the Physics and Astronomy
                             % Classification Scheme.
%Use showkeys class option if keyword
                              %display desired
\maketitle

\section{\label{Intro} Introduction}
Sch\"afke and Schmidt \cite{Schaefke}, \cite{Schmidt1}, and \cite{Schmidt2} studied two point connection problems between pairs of solutions around neighboring singularities using a contour integral approach based on the Cauchy Integral Formula. In \cite{Schmidt2}, they studied in particular the connection problem between pairs of solutions around regular singularities. They obtained expressions for the connection coefficients as a limit of a sequence involving the coefficients in the Frobenius expansion of the solution around 0. The shortcoming of this method is that it assumes these coefficients are known. For the Hypergeometric equation (see \cite{Erdely}), this is not a problem as the coefficients satisfy a two-term recurrence relation which is easy to solve. However this is not true for the Heun equation. The required coefficients are solutions of a three-term recurrence relation for which there is no known explicit solution in the general case. In this paper, we will modify the methods used by \cite{
Schmidt2} to fully solve the connection problem for a subclass of the Heun Equation for which this recurrence relation can be solved explicitly. We give explicit expressions for the connection coefficients. 
\\
\indent The Heun equation is an increasingly important equation which appears more and more frequently in the literature. Much of the work surrounding the Heun function involves finding integral representations. Several integral representations for the Hypergeometric function are known. These provide a succesful strategy for solving the two-point connection problem for the Hypergeometric equation (see \cite{Erdely} for details). In this paper we solve the two-point connection for a subclass of Heun equation without using any integral representations, thus illustrating the power of the strategy employed by Sch\"afke and Schmidt. 
\\
\indent We consider the two-point connection problem for the Huen equation \cite{Ronveaux} given by 
\begin{equation}\label{heun}
\frac{d^2y}{dz^2} + \left(\frac{\gamma}{z} + \frac{\delta}{z-1} + \frac{\epsilon}{z-a} \right)\frac{dy}{dz} + \frac{\alpha\beta z - q}{z(z-1)(z-a)}y = 0,
\end{equation}
with $a, q\in\mathbb{C}$ and where $\epsilon = \alpha+\beta+1-\gamma-\delta$ and $a \neq 0, 1$. It is well known that equation (\ref{heun}) has regular singularities at $0$, $1$, $a$, and $\infty$. Furthermore, equation (\ref{heun}) has a Frobenius solution which is regular for $|z|<\min\{1,|a|\}$ and is denoted $Hl(a, q; \alpha, \beta, \gamma, \delta; z)$, the local Heun function. Note that $Hl$ is normalized so that $Hl(a, q; \alpha, \beta, \gamma, \delta; 0) = 1$. The coefficients in the expansion 
\[
Hl(a, q; \alpha, \beta, \gamma, \delta; z) = \sum\limits_{k=0}^\infty A_k z^k, \quad |z|<\min\{1,|a|\}
\]
satisfy the well known \cite{Ronveaux} recurrence relation
\begin{equation} \label{recurrence}
 \begin{aligned}
  0 &= aA_1\gamma-qA_0, \\
  0 &= aP_k A_{k+1} - [Q_k+q]A_k+R_kA_{k-1}, \hspace{1cm} k \geq 1,
 \end{aligned}
\end{equation}
where $P_k = (k+1)(k+\gamma)$, $Q_k = k(k-1+\gamma)(1+a)+k(a\delta+\epsilon)$, $R_k = (k-1+\alpha)(k-1+\beta)$, and $A_0 = 1$.
Maier \cite{Maier} described the fundamental pairs of local Frobenius solutions to equation (\ref{heun}) and gave various relations satisfied by the local Heun function. We donote these pairs of solutions by $\{y_{01},y_{02}\}$, $\{y_{11},y_{12}\}$, $\{y_{a1},y_{a2}\}$, and $\{y_{\infty 1},y_{\infty 2}\}$ where
\begin{align*}
y_{01}(z) &= Hl(a, q; \alpha, \beta, \gamma, \delta; z), \\
y_{02}(z) &= z^{1-\gamma}Hl(a,q+(1-\gamma)(\alpha+\beta+1-\gamma +(a-1)\delta);1+\alpha-\gamma,1+\beta-\gamma,2-\gamma,\delta;z), \\
&
\\
y_{11}(z) &= Hl(1-a,\alpha\beta - q;\alpha,\beta,\delta,\gamma;1-z), \\
y_{12}(z) &= (1-z)^{1-\delta}Hl(1-a,\alpha\beta-q+(1-\delta)(\epsilon+(1-a)\gamma);1+\alpha-\delta;1+\beta-\delta,2-\delta,\gamma;1-z)\\
&\\
y_{a1}(z) &= Hl\left(\frac{a}{a-1},\frac{a\alpha\beta - q}{a-1};\alpha,\beta,\alpha+\beta+1-\gamma-\delta,\delta;\frac{z-a}{1-a}\right), \\
y_{a2}(z) &= (z-a)^{1-\epsilon} Hl\left(\frac{a}{a-1},\frac{-q+a(\gamma+\delta-\alpha)(\gamma+\delta-\beta)+\gamma(\epsilon-1)}{a-1};\gamma+\delta-\beta,\gamma+\delta-\alpha,2-\epsilon,\delta;\frac{z-a}{1-a}\right), \\
& \\
y_{\infty 1}(z) &= z^{-\alpha}Hl\left(\frac{1}{a},\frac{q+\alpha(a(\alpha+1-\gamma-\delta)+\delta-\beta)}{a}; \alpha,\alpha+1-\gamma,1+\alpha-\beta,\delta;z^{-1}\right), \\
y_{\infty 2}(z) &= z^{-\beta}Hl\left(\frac{1}{a},\frac{q+\beta(a(\beta+1-\gamma-\delta)+\delta-\alpha)}{a};\beta,\beta+1-\gamma,1+\beta-\alpha,\delta;z^{-1} \right).
\end{align*}
In order for these pairs to be linearly independent and  well-defined, we require that $\gamma,~\delta,~\epsilon,~\alpha -\beta \notin \mathbb{Z}$.

\section{Preliminaries}\label{pro}
In this paper we consider the subclass of equation (\ref{heun}) where $\delta=(\alpha+\beta+1-\gamma)/2$, $q=0$, and $a=-1$. That is, we consider the Fuchsian equation
\begin{equation} \label{subClass}
\frac{d^2y}{dz^2} + \left( \frac{\gamma}{z} + \frac{\delta}{z-1} + \frac{\delta}{z+1}\right)\frac{dy}{dz} + \frac{\alpha\beta }{(z-1)(z+1)}y = 0.
\end{equation}
\begin{remark}
 Note that for this subclass of the Heun Equation, the methods employed in \cite{Schmidt2} cannot be immediately applied since they assumed that the differential equation has no other singularity in the closed disk $\{z\in\mathbb{C}: |z|\leq 1\}$ besides 0 and 1, whereas (\ref{subClass}) has singularities at $-1$, $0$, and $1$.
\end{remark}

It is not difficult to see that (\ref{recurrence}) becomes
\begin{equation*}
A_{k+1} = \frac{(k-1+\alpha)(k-1+\beta)}{(k+1)(k+\gamma)}A_{k-1}, \quad k \geq 1.
\end{equation*}
Whence we obtain
\begin{equation}\label{coeffs}
A_{2n} = \frac{(\frac{\alpha}{2})_n(\frac{\beta}{2})_n}{n!(\frac{\gamma+1}{2})_n}, \quad n \geq 0,
\end{equation}
and $A_{2n+1} = 0,~\forall n \geq 0$. Hence,
\begin{equation}
Hl(-1,0;\alpha,\beta,\gamma,\delta;z) = \sum\limits_{n=0}^\infty \frac{(\frac{\alpha}{2})_n(\frac{\beta}{2})_n}{n!(\frac{\gamma+1}{2})_n} z^{2n}.
\end{equation}
Equation (\ref{subClass}) has fundamental pairs of solutions given by 
\begin{equation} \label{sols1}
\begin{aligned}
&y_{01}(z) = Hl(-1,0;\alpha,\beta,\gamma,\delta;z) \\
&y_{02}(z) = z^{1-\gamma}Hl(-1,0;1+\alpha-\gamma,1+\beta-\gamma,2-\gamma,\delta;z), \\
& \\
&y_{+1}(z) = Hl(2,\alpha\beta;\alpha,\beta,\delta,\gamma;1-z),\\
&y_{+2}(z) = (1-z)^{1-\delta}Hl(2,\alpha\beta+(1-\delta)(\delta+2\gamma);1+\alpha-\delta;1+\beta-\delta,2-\delta,\gamma;1-z),\\
& \\
&y_{-1}(z) = Hl(1/2,\alpha\beta/2;\alpha,\beta,\alpha+\beta+1-\gamma-\delta,\delta;(z+1)/2),\\
&y_{-2}(z) = (z+1)^{1-\delta}Hl(1/2,[(\gamma+\delta-\alpha)(2\gamma+\delta-\beta)-\gamma\beta]/2;\gamma+\delta-\beta,\gamma+\delta-\alpha,2-\delta,\delta;(z+1)/2),\\
& \\
&y_{\infty 1}(z) = z^{-\alpha}Hl\left(-1,0; \alpha,\alpha+1-\gamma,1+\alpha-\beta,\delta;z^{-1}\right),\\
&y_{\infty 2}(z) = z^{-\beta}Hl\left(-1,0;\beta,\beta+1-\gamma,1+\beta-\alpha,\delta;z^{-1} \right). 
\end{aligned}
\end{equation}
\begin{remark}\label{extension}
 Note that any solution of (\ref{heun}) may be analytically continued along any path in $\mathbb{C}-\{0,1,a\}$ and the anayltic continuation is a solution (see for example Theorem 3.7.2 in \cite{Ablo} or 10.1 in \cite{Cop}). If two paths are homotopic then the continuation is unique by the Monodromy Theorem (see \cite{Ahlfors} or \cite{Conway}). Thus, if the domain $D \subset \mathbb{C}-\{0,1,a\}$ is simply connected and has non-empty intersection with the open disc $\{z\in\mathbb{C}:|z|<1\}$, then in particular it is not difficult to see that $y_{01}$ has a unique analytic extension to D.
\end{remark}
We will also find the following results helpful in proving our main result.

\begin{lemma}\label{beta_function}
If $\Re{\alpha}$, $\Re{\beta}>0$, then
\[
\int_0^1 t^{\alpha-1}(1-t)^{\beta-1}~dt = \frac{\Gamma(\alpha)\Gamma(\beta)}{\Gamma(\alpha + \beta)}.
\]
\end{lemma}
\begin{proof}
 This is a standard result about the Euler-Beta Function. For a proof, see, for example, Section 1.5 in \cite{Erdely}.
\end{proof}
\begin{lemma} \label{ratio} {\bf(Asymptotics of ratio of two gamma functions)} \newline
In the intersection of the sectors $|\arg (z+a)|<\pi$ and $|\arg (z+b)|<\pi$, we have
\[
\frac{\Gamma(z+a)}{\Gamma(z+b)} \sim z^{a-b}, \quad z\to \infty
\]
\end{lemma}
\begin{proof}
 This is a standard result which may be readily derived from Stirling's Series. For an alternative proof, see pg. 118 in \cite{Olver}.
\end{proof}
\begin{lemma} \label{integral}
Let $1<\rho<2$ and $\alpha \in \mathbb{C}$. If $k>\Re\alpha$ and $\Re\alpha>0$, then
\begin{equation*}
\int_1^\rho \! (z-1)^{\alpha}z^{-k-1} \mathrm{d}z = \frac{\Gamma(k-\alpha)\Gamma(\alpha+1)}{\Gamma(k+1)}+\mathcal{O}(\rho^{-k}) , \quad k\to \infty.
\end{equation*}
Let $F:B_1(1)\to \mathbb{C}$ be holomorphic. Then, 
\begin{equation*}
\int_1^{\rho} \! (z-1)^{\alpha}z^{-k-1}F(z) \mathrm{d}z = \mathcal{O}(k^{-\Re\alpha-1}), \quad k\to \infty.
\end{equation*}
where the powers in the above integral take their principal values.
\end{lemma}
\begin{proof} Observe that
\[
\int_1^\rho \! (z-1)^{\alpha}z^{-k-1} \mathrm{d}z = \int_1^\infty \! (z-1)^{\alpha}z^{-k-1} \mathrm{d}z - \int_\rho^\infty \! (z-1)^{\alpha}z^{-k-1} \mathrm{d}z.
\]
Since $k >\Re\alpha$ and $\Re\alpha>0$, using the transformation $z = 1/t$ we obtain
\[
\left| \int_\rho^\infty \! (z-1)^{\alpha}z^{-k-1} \mathrm{d}z \right| = \left| \int_0^{\frac{1}{\rho}} \! t^{k-1-\alpha}(1-t)^{\alpha} \mathrm{d}t \right| \leq  \int_0^{\frac{1}{\rho}} \! x^{k-1-\Re\alpha} \mathrm{d}x = \frac{\rho^{-k+\Re\alpha}}{k-\Re\alpha}.
\]
Hence,
\begin{equation*}
\int_\rho^\infty \! (z-1)^{\alpha}z^{-k-1} \mathrm{d}z = \mathcal{O}(\rho^{-k}), \quad k\to \infty.
\end{equation*}
Also, 
\[
\int_1^\infty \! (z-1)^{\alpha}z^{-k-1} \mathrm{d}z \stackrel{z=1/t}{=} \int_0^1 \! (1-t)^{\alpha}t^{k-\alpha-1} \mathrm{d}t \stackrel{\ref{beta_function}}{=} \frac{\Gamma(k-\alpha)\Gamma(\alpha+1)}{\Gamma(k+1)}.
\]
We prove now the second part of the lemma. Notice that since $F$ is holomorphic we may find an $M \in \mathbb{R}^+$ such that
\begin{align*}
\left|\int_1^{\rho} \! (z-1)^{\alpha}z^{-k-1}F(z) \mathrm{d}z \right| &\leq M \int_1^\rho \! x^{-k-1}(x-1)^{\Re\alpha} \mathrm{d}x \leq M \int_1^\infty \! x^{-k-1}(x-1)^{\Re\alpha} \mathrm{d}x, \\
&\stackrel{x=1/t}{=} M \int_0^1 \! t^{k-1-\Re\alpha}(1-t)^{\Re\alpha} \mathrm{d}t \stackrel{\ref{beta_function}}{=} M \frac{\Gamma(k-\Re\alpha)\Gamma(\Re\alpha+1)}{\Gamma(k+1)}.
\end{align*}
Using Lemma~\ref{ratio}, we get
\begin{equation*}
\int_1^{\rho} \! (z-1)^{\alpha}z^{-k-1}F(z) \mathrm{d}z = \mathcal{O}(k^{-\Re\alpha-1}), \quad k\to\infty.
\end{equation*}
This concludes the proof. ~~$\square$
\end{proof}

\section{Solution of The Two-Point Connection Problem}\label{sol}
We consider simultaneously the two-point connection problem between 0 and 1 and between 0 and -1. That is we seek coefficients $c^+_{11},~c^+_{12},~c^-_{11},~c^-_{12}$ such that 
\begin{eqnarray} \label{connection}
y_{01} &=& c^+_{11}y_{+1} + c^+_{12}y_{+2} \label{connection1} \\
y_{01} &=& c^-_{11}y_{-1} + c^-_{12}y_{-2} \label{connection2}
\end{eqnarray}
Note however, that $c^+_{11}$ and $c^-_{11}$ may be easily found. We recall the well-known result \cite{Erdely}
\[
\lim\limits_{z\to 1+} F(\alpha,\beta,\gamma;z)=\frac{\Gamma(\gamma-\alpha-\beta)\Gamma(\gamma)}{\Gamma(\gamma-\alpha)\Gamma(\gamma-\beta)}
\]
Using this relation and taking the limit of (\ref{connection1}) as $z \to 1$ and assuming $\Re(1-\delta)>0$ we obtain
\[
c^+_{11} = \sum\limits_{n=0}^\infty \frac{(\frac{\alpha}{2})_n(\frac{\beta}{2})_n}{n!(\frac{\gamma+1}{2})_n} = \lim\limits_{z\to 1+} F\left(\frac{\alpha}{2},\frac{\beta}{2},\frac{\gamma+1}{2};z\right) = \frac{\Gamma(\frac{\gamma+1}{2}-\frac{\alpha}{2}-\frac{\beta}{2})\Gamma(\frac{\gamma+1}{2})}{\Gamma(\frac{\gamma+1}{2}-\frac{\alpha}{2})\Gamma(\frac{\gamma+1}{2}-\frac{\beta}{2})},
\]
and similarly taking the limit as $z \to -1$ of (\ref{connection2}) we obtain
\[
c^-_{11} = \frac{\Gamma(\frac{\gamma+1}{2}-\frac{\alpha}{2}-\frac{\beta}{2})\Gamma(\frac{\gamma+1}{2})}{\Gamma(\frac{\gamma+1}{2}-\frac{\alpha}{2})\Gamma(\frac{\gamma+1}{2}-\frac{\beta}{2})} = c^+_{11}.
\]
We compute $c^+_{12}$ and $c^-_{12}$ with the following theorem.
\begin{theorem}
Let $y_{01},~y_{+1},~y_{+2},~y_{-1},~y_{-2}$ be as in (\ref{sols1}) and $c^+_{11},~c^+_{12},~c^-_{11},~c^-_{12}$ be as in (\ref{connection1}) and (\ref{connection2}).Furthermore let $\Re(1-\delta)>0$, then we have
\[
c^+_{12} = 2^{1-\delta}\frac{\Gamma(\delta-1)\Gamma(\frac{\gamma+1}{2})}{\Gamma(\frac{\alpha}{2})\Gamma(\frac{\beta}{2})} = c^-_{12},
\]
\end{theorem}
\begin{proof}
For simplification, first observe that $y_{01}$ is of the form
\[
y_{01}(z) = \sum\limits_{k=0}^\infty d_k z^k.
\]
\begin{figure}[!ht]
\begin{center}
\includegraphics[scale = 0.75]{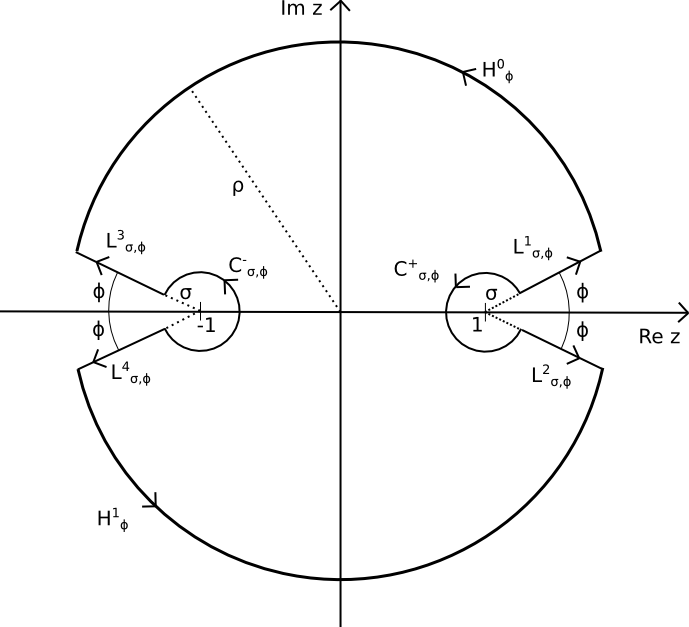}
\caption{Integration contour showing components of $C_{\sigma,\phi}$}\label{Fig1}
\end{center}
\end{figure}
Let $C_{\sigma,\phi} = C^0_{\sigma,\phi}+L^1_{\sigma,\phi}-C^+_{\sigma,\phi}-L^2_{\sigma,\phi}+L^4_{\sigma,\phi}-C^-_{\sigma,\phi}-L^3_{\sigma,\phi}$ be the contour shown in FIG. \ref{Fig1} where $C^0_{\sigma,\phi}= H^0_{\sigma,\phi}+H^1_{\sigma,\phi}$ and $1<\rho<2$. By the Cauchy Integral Formula we have for any $k\in \mathbb{N}_0$, and $\sigma,~\phi >0$ sufficiently small
\begin{equation}\label{1}
d_k = \frac{1}{2\pi i }\int_{C_{\sigma,\phi}} z^{-k-1} \hat{y}_{01}(z) dz
\end{equation}
where $\hat{y}_{01}$ is the unique analytic extension of $y_{01}$ to the simply connected set $\mathbb{C}-((-\infty, -1] \cup [1,+\infty))$ guaranteed to exist by Remark~\ref{extension}. In particular,
\[
\hat{y}_{01}(z) = c^+_{11}y_{+1}(z) + c^+_{12}y_{+2}(z), \quad z\in L^1_{\sigma,\phi} \cup C^+_{\sigma,\phi} \cup L^2_{\sigma,\phi}
\]
\[
\hat{y}_{01}(z) = c^-_{11}y_{-1}(z) + c^-_{12}y_{-2}(z), \quad z\in L^3_{\sigma,\phi} \cup C^-_{\sigma,\phi} \cup L^4_{\sigma,\phi}
\]
where we take the principle value of the powers occuring in $y_{+2}$ and $y_{-2}$. Notice that the left hand side of (\ref{1}) above does not depend on $\sigma$ or $\phi$. Hence we consider the limit of the expression on the right hand side as $\sigma,~\phi \to 0$. This limit if it exists should be equal to $d_k$. So
\[
2\pi i d_k = \lim\limits_{\sigma \to 0} \left(\lim\limits_{\phi \to 0} \left(\int_{C_{\sigma,\phi}} z^{-k-1} \hat{y}_{01}(z) dz \right) \right) = I^0_k + I^1_k + I^2_k
\]
where
\[
I^0_k = I^0_{k,1} - I^0_{k,2} - I^0_{k,3}  
\]
\[
I^1_k = \lim\limits_{\sigma \to 0} \left(\lim\limits_{\phi \to 0} \left(\int_{L^1_{\sigma,\phi}} z^{-k-1} (c^+_{11}y_{+1} + c^+_{12}y_{+2})(z) dz-\int_{L^2_{\sigma,\phi}} z^{-k-1} (c^+_{11}y_{+1} + c^+_{12}y_{+2})(z) dz \right)\right)
\]
\[
I^2_k = \lim\limits_{\sigma \to 0} \left( \lim\limits_{\phi \to 0} \left(\int_{L^4_{\sigma,\phi}} z^{-k-1} (c^-_{11}y_{-1} + c^-_{12}y_{-2})(z) dz - \int_{L^3_{\sigma,\phi}} z^{-k-1} (c^-_{11}y_{-1} + c^-_{12}y_{-2})(z) dz\right)\right)
\]
and
\[
 I^0_{k,1} = \lim\limits_{\sigma \to 0} \left(\lim\limits_{\phi \to 0} \left(\int_{C^0_{\sigma,\phi}} z^{-k-1} \hat{y}_{01}(z) dz\right)\right), \quad I^0_{k,2} = \lim\limits_{\sigma \to 0} \left(\lim\limits_{\phi \to 0} \left( \int_{C^+_{\sigma,\phi}} z^{-k-1} (c^+_{11}y_{+1} + c^+_{12}y_{+2})(z) dz \right)\right) 
\]
\[
 I^0_{k,3} = \lim\limits_{\sigma \to 0} \left(\lim\limits_{\phi \to 0} \left(\int_{C^-_{\sigma,\phi}} z^{-k-1} (c^-_{11}y_{-1} + c^-_{12}y_{-2})(z) dz \right)\right)
\]
First we deal with $I^0_k$. In particular, we will show that $I^0_k = \mathcal{O}(\rho^{-k})$. Note that we have the following parametrizations 
\begin{align*}
 &C^+_{\sigma,\phi} : z(\theta) = 1 + \sigma e^{i\theta}, \quad \phi \leq \theta \leq 2\pi - \phi \\
 &C^-_{\sigma,\phi} : z(\theta) = -1 + \sigma e^{i\theta}, \quad \phi - \pi \leq \theta \leq \pi - \phi
\end{align*}
Since $y_{+1}$ is holomorphic in $\{z:|z-1|<1\}$ and $y_{-1}$ is holomorphic in $\{z:|z+1|<1\}$ we obtain
\[
 I^0_{k,2} =  c^+_{12} \lim\limits_{\sigma \to 0} \left(\lim\limits_{\phi \to 0} \left(\int_{C^+_{\sigma,\phi}} z^{-k-1} y_{+2}(z) dz \right)\right), \quad
 I^0_{k,3} = c^-_{12} \lim\limits_{\sigma \to 0} \left(\lim\limits_{\phi \to 0} \left(\int_{C^-_{\sigma,\phi}} z^{-k-1} y_{-2}(z) dz\right)\right)
\]
We express $y_{+2}$ and $y_{-2}$ as $y_{+2}(z) = (1-z)^{1-\delta}f_1(z)$ and $y_{-2}(z) = (z+1)^{1-\delta}f_2(z)$ where $f_1,~f_2$ are holomorphic functions in the discs of radius 1 centered at $1,~-1$ respectively. Thus, we obtain
\[
 I^0_{k,2} = c^+_{12} \lim\limits_{\sigma \to 0} \left(\lim\limits_{\phi \to 0} \left(\int_{C^+_{\sigma,\phi}} z^{-k-1} (1-z)^{1-\delta}f_1(z) dz\right)\right)
\]
Let M be such that $|f_1(z)|,|f_2(z)| < M$, $\forall z\in \{z\in\mathbb{C}: |z-1|<1\}$ (M is guaranteed to exist since $F_1$ is holomorphic in an open disc centered at 1), then 
\[
 \left|\int_{C^+_{\sigma,\phi}} z^{-k-1} (1-z)^{1-\delta}f_1(z) dz \right| = \left|\int_{\phi}^{2\pi-\phi}(1+\sigma e^{i\theta})^{-k-1}(-\sigma e^{i\theta})^{1-\delta}f_1(1+\sigma e^{i\theta})i\sigma e^{i\theta}d\theta \right| 
\]
\[
 \leq \int_{\phi}^{2\pi-\phi} M |1+\sigma e^{i\theta}|^{-k-1} \sigma^{2-\Re\delta} d\theta \leq M 2^{k+1} \int_{\phi}^{2\pi-\phi} \sigma^{2-\Re\delta} d\theta = M 2^{k+2}(\pi-\phi)\sigma^{2-\Re\delta}
\]

The second inequality follows from the fact that for $\sigma$ sufficiently small $|1+\sigma e^{i\theta}|>1/2$. Since $\Re(1-\delta)>0$, this implies that $I^0_{k,2} = 0$. It may similarly be shown that $I^0_{k,3} = 0$. Note that $y_{01}$ may be extended analytically to simply connected domains $D_1$ and $D_2$ containing $H^0_0$ and $H^1_0$ respectively. By the continuity of these extensions, their absolute values have a common upper bound $M\in \mathbb{R}^+$ on $H^0_0$ and $H^1_0$. Since these extensions also extend $\hat{y}_{01}$ and since $H^0_\phi \subset H^0_0$ and $H^1_\phi \subset H^1_0$, this bound also holds for $z \in H^0_\phi \cup H^1_\phi = C^0_{\sigma, \phi}$. Thus by the M-L Formula (see for example (9) page 83 in \cite{Ahlfors}) we have
\[
 \left|\int_{C^0_{\sigma,\phi}} z^{-k-1} \hat{y}_{01}(z) dz \right| \leq 2\pi M \rho^{-k} 
\]
Hence $I^0_k = I^0_{k,1} = \mathcal{O}(\rho^{-k})$. We now give parametrizations of the contours $L^1_{\sigma,\phi}$, $L^2_{\sigma,\phi}$, $L^3_{\sigma,\phi}$, and $L^4_{\sigma,\phi}$.
\[
L^1_{\sigma,\phi} : z(r) = 1 + re^{i\phi}, \quad L^2_{\sigma,\phi} : z(r) = 1 + re^{-i\phi}, \quad L^3_{\sigma,\phi} : z(r) = -1 - re^{-i\phi}, \quad L^4_{\sigma,\phi} : z(r) = -1 - re^{i\phi}
\]
where the parameter $r$ runs $\sigma \leq r \leq \sqrt{\cos^2\phi +\rho^2-1} - \cos\phi$. Using the above parametrizations, and taking the limits we see that
\[
 I^1_k = c^+_{12}\int_{0}^{\rho-1} (1+r)^{-k-1} (e^{-2\pi i(1-\delta)}-1)y_{+2}(1+r)dr \stackrel{x=1+r}{=} c^+_{12}\left(1-e^{-2\pi i(1-\delta)}\right)J^1_k
\]
and similarly
\[
 I^2_k = c^-_{12}\int_{0}^{\rho-1} (-1-r)^{-k-1} (1-e^{-2\pi i(1-\delta)})y_{-2}(-1-r)dr \stackrel{x=-(1+r)}{=} c^-_{12}\left(1-e^{-2\pi i(1-\delta)}\right)J^2_k
\]
where 
\[
 J^1_k = \int_{\rho}^{1} x^{-k-1} y_{+2}(x)dx, \quad J^2_k = \int_{-\rho}^{-1} x^{-k-1} y_{-2}(x)dx
\]
Now, we rewrite $y_{+2}$ and $y_{-2}$ as follows
\begin{equation}\label{representation}
y_{+2}(z) = \left(\sum\limits_{j=0}^m G_j (1-z)^{1-\delta + j} + (1-z)^{2+m-\delta} F_1(z)\right), \quad y_{-2}(z) = \left(\sum\limits_{j=0}^m H_j (1+z)^{1-\delta + j} + (1+z)^{2+m-\delta} F_2(z)\right)
\end{equation}
where $F_1,~F_2$ are analytic functions in the discs of radii 1 centered at $1,~-1$ respectively and $G_0 = 1 = H_0$.
Using the representation for $y_{+2}$ and $y_{-2}$ found in (\ref{representation}), we obtain
\[
J^1_k = \left(\sum\limits_{j=0}^m G_j  \int_\rho^1 x^{-k-1} (1-x)^{1-\delta + j} dx     + \int_\rho^1 x^{-k-1} (1-x)^{2+m-\delta} F_1(x) dx\right)
\]
\[
J^2_k = \left(\sum\limits_{j=0}^m H_j \int_{-\rho}^{-1} x^{-k-1}(1+x)^{1-\delta + j}dx   + \int_{-\rho}^{-1} x^{-k-1} (1+x)^{2+m-\delta} F_2(x)dx \right)
\]
Hence if $\Re(1-\delta)>0$, we may apply Lemma~\ref{integral} to obtain 
\[
J^1_k = -\sum\limits_{j=0}^m G_j  \exp(\pi i(1-\delta + j))\frac{\Gamma(k-j-1+\delta)\Gamma(2-\delta + j)}{\Gamma(k+1)}+\mathcal{O}(\rho^{-k}) + \mathcal{O}(k^{-\Re(3+m-\delta)})
\]
and similarly
\begin{eqnarray*}
J^2_k &\stackrel{z=-v}{=}& \left(\sum\limits_{j=0}^m H_j (-1)^{-k-1} \int_1^{\rho} v^{-k-1}(1-v)^{1-\delta + j}dv + (-1)^{-k-1}\int_1^{\rho} v^{-k-1} (1-v)^{2+m-\delta} F_2(-v)dv \right) \\
&=& \sum\limits_{j=0}^m H_j (-1)^{-k-1} \exp(\pi i(1-\delta + j)) \frac{\Gamma(k-j-1+\delta)\Gamma(2-\delta + j)}{\Gamma(k+1)} +\mathcal{O}(\rho^{-k}) + \mathcal{O}(k^{-\Re(3+m-\delta)})
\end{eqnarray*}
Hence, when we multiply by the ratio $\Gamma(k+1)/\Gamma(k-1+\delta)$
\begin{eqnarray*}
\frac{\Gamma(k+1)}{\Gamma(k-1+\delta)}d_k &=& \mathcal{O}(\rho^{-k}) + \frac{\Gamma(k+1)}{\Gamma(k-1+\delta)}\frac{1}{2\pi i} \left[I^1_k + I^2_k\right] \\
&=& \mathcal{O}(\rho^{-k}) + \frac{\Gamma(k+1)}{\Gamma(k-1+\delta)}\frac{1}{2\pi i} \left[c^+_{12}L^1_k + c^-_{12}L^2_k\right] \\
&\stackrel{\ref{ratio}}{=}& \mathcal{O}(\rho^{-k}) +\mathcal{O}(k^{-m-1})+ \frac{\Gamma(k+1)}{\Gamma(k-1+\delta)} \frac{[1-\exp(-2\pi i(1-\delta))]}{2\pi i} \left[c^+_{12} J^1_k + c^-_{12} J^2_k \right] \\
&=& \frac{\sin[\pi(\delta-1)]}{\pi} \left[c^+_{12} \sum\limits_{j=0}^m \frac{G_j e^{j\pi i}\Gamma(2-\delta+j)}{\prod_{\sigma=1}^j(k-1+\delta-\sigma)} + (-1)^{-k} c^-_{12} \sum\limits_{j=0}^m \frac{H_j e^{j\pi i}\Gamma(2-\delta+j)}{\prod_{\sigma=1}^j(k-1+\delta-\sigma)} \right]+ \\
 &+& \mathcal{O}(\rho^{-k}) +\mathcal{O}(k^{-m-1}) \\ 
\end{eqnarray*}
If $k=2n$ (i.e. k even), we have when we take the limit $n\to \infty$ and using the gamma reflection formula (see (6) at page 3 in \cite{Erdely})
\[
\lim\limits_{n \to \infty} \frac{\Gamma(2n+1)}{\Gamma(2n-1+\delta)}d_{2n} = \frac{(c^+_{12} + c^-_{12})}{\Gamma(\delta-1)}
\]
Also, taking $k=2n+1$ (i.e. k odd) we obtain 
\[
\lim\limits_{n \to \infty} \frac{\Gamma(2n+2)}{\Gamma(2n+\delta)}d_{2n+1} = \frac{(c^+_{12} - c^-_{12})}{\Gamma(\delta-1)}
\]
Using (\ref{coeffs}) we obtain
\[
c^+_{12} + c^-_{12} = \lim\limits_{n \to \infty} \frac{\Gamma(2n+1)\Gamma(\delta-1)}{\Gamma(2n-1+\delta)}\frac{(\frac{\alpha}{2})_n(\frac{\beta}{2})_n}{n!(\frac{\gamma+1}{2})_n}, \quad c^+_{12} - c^-_{12} = 0 \implies c^+_{12} = c^-_{12}
\]
Furthermore using the fact that $(\alpha)_n = \Gamma(\alpha+n)/\Gamma(\alpha)$ we obtain 
\begin{eqnarray*}
c^+_{12} + c^-_{12} &=& \frac{\Gamma(\delta-1)\Gamma([\gamma+1]/2)}{\Gamma(\alpha/2)\Gamma(\beta/2)}\lim\limits_{n \to \infty} \frac{\Gamma(2n+1)\Gamma(\alpha/2+n)\Gamma(\beta/2+n)}{\Gamma(2n-1+\delta)\Gamma(n+1)\Gamma([\gamma+1]/2+n)} \\
&\stackrel{\ref{ratio}}{=}& 2^{2-\delta}\frac{\Gamma(\delta-1)\Gamma([\gamma+1]/2)}{\Gamma(\alpha/2)\Gamma(\beta/2)}
\end{eqnarray*}
This completes the proof.~~$\square$
\end{proof}

\section{The Connection Matrices}\label{connection_matrices}
Maier \cite{Maier} showed the following relations hold
\begin{equation*}
\begin{aligned}
&y_{11}(z) = z^{1-\gamma}Hl(1-a,-q+\alpha\beta+(\gamma-1)(1-a)\delta;1+\alpha-\gamma,1+\beta-\gamma,\delta,2-\gamma;1-z), \\
&y_{a1}(z) = \left(\frac{z}{a}\right)^{1-\gamma}Hl\left(\frac{a}{a-1},Q;1+\alpha-\gamma,1+\beta-\gamma,\alpha+\beta+1-\gamma-\delta,\delta;\frac{z-a}{1-a} \right),
\end{aligned}
\end{equation*}
where
\[
Q = \frac{a(1+\alpha-\gamma)(1+\beta-\gamma)-q-(1-\gamma)(\alpha+\beta+1-\gamma+(a-1)\delta)}{a-1},
\]
and
\[
y_{11}(z) = z^{-\alpha} Hl\left(1-\frac{1}{a},\frac{-q+\alpha[(a-1)\delta+\beta]}{a};\alpha,\alpha+1-\gamma,\delta,1+\alpha-\beta;1-z^{-1}\right),
\]
Furthermore, in the special case $a=-1$, we assert that
\[
y_{a1}(z) = (-z)^{-\alpha}Hl\left(\frac{1}{1-a},\frac{-q+\alpha[(\alpha+1-\gamma-\delta)(1-a)+\beta]}{1-a};\alpha,\alpha+1-\gamma,\epsilon,\delta;\frac{z^{-1}-a}{1-a}\right).
\]
Indeed the transformation $z=1/t$, $w(z) = f(t)$, and $f(t) = t^{\alpha}\phi(t)$ is a symmetry of equation~(\ref{heun}). Applying the transformation, (\ref{heun}) is transformed into a Heun equation with singularities at $\{0,1,1/a,\infty\}$ and that the r.h.s. above is a solution to (\ref{heun}) in a neighborhood of $1/a$. In particular if $a=-1$ then we have another solution in a neighborhood of $-1$. Hence the r.h.s. above must be a linear combination of $y_{a1}$ and $y_{a2}$ in the case $a=-1$. Comparing the behaviours of the functions, it becomes clear that the assertion above is true.
Let
\[
q_1(\alpha,\beta,\gamma) = \frac{\Gamma(\frac{\gamma+1}{2}-\frac{\alpha}{2}-\frac{\beta}{2})\Gamma(\frac{\gamma+1}{2})}{\Gamma(\frac{\gamma+1}{2}-\frac{\alpha}{2})\Gamma(\frac{\gamma+1}{2}-\frac{\beta}{2})}, \quad
q_2(\alpha,\beta,\gamma) = 2^{1-\delta}\frac{\Gamma(\delta-1)\Gamma(\frac{\gamma+1}{2})}{\Gamma(\frac{\alpha}{2})\Gamma(\frac{\beta}{2})},
\]
Then
\begin{equation*}
\begin{aligned}
&Hl(-1,0;\alpha,\beta,\gamma,\delta;z) = q_1(\alpha,\beta,\gamma)Hl(2,\alpha\beta;\alpha,\beta,\delta,\gamma;1-z) + \\
&q_2(\alpha,\beta,\gamma)(1-z)^{1-\delta}Hl(2,\alpha\beta+(1-\delta)(\delta+2\gamma);1+\alpha-\delta;1+\beta-\delta,2-\delta,\gamma;1-z),
\end{aligned}
\end{equation*}
and
\begin{equation*}
\begin{aligned}
&Hl(-1,0;\alpha,\beta,\gamma,\delta;z) = q_1(\alpha,\beta,\gamma)Hl\left(\frac{1}{2},\frac{\alpha\beta}{2};\alpha,\beta,\delta,\delta;\frac{z+1}{2}\right) + \\
&q_2(\alpha,\beta,\gamma)\left(\frac{z+1}{2}\right)^{1-\delta}Hl\left(\frac{1}{2},\frac{\alpha\beta}{2}+(1-\delta)\left(\gamma+\frac{\delta}{2} \right);1+\alpha-\delta,1+\beta-\delta,2-\delta,\delta;\frac{z+1}{2}\right).
\end{aligned}
\end{equation*}
Using the above relations it follows, after some computations, that the connection matrices we seek are
\[
C_{0+} = \begin{pmatrix} q_1(\alpha, \beta, \gamma) & q_2(\alpha, \beta, \gamma) \\ q_1(1+\alpha-\gamma, 1+\beta-\gamma, 2-\gamma) & q_2(1+\alpha-\gamma, 1+\beta-\gamma, 2-\gamma)\end{pmatrix},
\]
\[
C_{0-} = \begin{pmatrix} q_1(\alpha, \beta, \gamma) & q_2(\alpha, \beta, \gamma) \\ (-)^{1-\gamma}q_1(1+\alpha-\gamma, 1+\beta-\gamma, 2-\gamma) & (-)^{1-\gamma}q_2(1+\alpha-\gamma, 1+\beta-\gamma, 2-\gamma)\end{pmatrix},
\]
\[
C_{\infty +} = \begin{pmatrix} q_1(\alpha, \alpha+1-\gamma, 1+\alpha-\beta) & (-)^{\delta-1}q_2(\alpha, \alpha+1-\gamma, 1+\alpha-\beta) \\ q_1(\beta, \beta+1-\gamma, 1+\beta-\alpha) & (-)^{\delta-1}q_2(\beta, \beta+1-\gamma, 1+\beta-\alpha)\end{pmatrix},
\]
\[
C_{\infty -} = \begin{pmatrix} (-)^{-\alpha}q_1(\alpha, \alpha+1-\gamma, 1+\alpha-\beta) & (-)^{\delta-\alpha-1}q_2(\alpha, \alpha+1-\gamma, 1+\alpha-\beta) \\ (-)^{-\beta}q_1(\beta, \beta+1-\gamma, 1+\beta-\alpha) & (-)^{\delta-\beta-1}q_2(\beta, \beta+1-\gamma, 1+\beta-\alpha)\end{pmatrix}.
\]
\begin{remark}
 These matrices remarkably provide a way for us to express $y_{+1}$ and $y_{-1}$ linearly in terms of $y_{01}$ and $y_{02}$. This is interesting because the coefficients appearing in $y_{+1}$ and $y_{-1}$ satisfy a much more general three-term recurrence relation 
 (and perhaps much more difficult to solve) than that of the coefficients of $y_{01}$ and $y_{02}$. Hence it would be difficult to express $y_{+1}$ and $y_{-1}$ in closed form if one only had (\ref{recurrence}) to rely on. In actuality, our matrices allow $y_{+1}$ and $y_{-1}$ to be expressed in closed form in terms of $y_{01}$ and $y_{02}$, which have closed form expressions. 
\end{remark}

\section{Conclusion}
In summary, we have modified the methods used by \cite{Schmidt2} to explicitly solve the two-point connection problem for a subclass of the Heun equation. To the best of our knowledge, this is the first time explicit expressions for the connection coefficients for even a subclass of the Heun equation has been given. We emphasize that the subclass of the Heun equation we have considered has 3 free paramters (just as many as the Hypergeometric equation). Our results will likely have many applications in a diverse range of fields, notably in mathematical physics where the solutions of important equations may be given in terms of Heun functions. The connection matrices we have found here will enable the construction of global analytic solutions of such equations. Take, for example, the work of \cite{BWN} where the most general class of potential was given such that the solution of the one-dimensional Schr\"odinger equation may be expressed in terms of Heun functions. Our results should enable the computation of 
bound states and energy eigenvalues, and the study of scattering and tunneling phenomena for a subclass of the potentials derived by \cite{BWN}. 

\begin{center} \large{{\bf Acknowledgements}} \end{center}
One of the authors (R. Williams) would like to thank Prof. Fabio Vlacci of the Mathematics Department, University of Florence (Florence, Italy) for some discussions they had at the International Center for Theoretical Physics (Trieste, Italy).

\bibliography{refs}% Produces the bibliography via BibTeX.

\end{document}